\documentclass[letterpaper,12pt]{article}

\usepackage[intlimits]{amsmath}
\usepackage{amssymb}
\usepackage{amsthm}
\usepackage{mathrsfs}
\usepackage{cite}
\usepackage{color}
\usepackage{bm}

\newcommand{\la}{\langle}
\newcommand{\ra}{\rangle}
\newcommand{\pr}{\partial}
\newcommand{\dom}{\Omega}

\newcommand{\im}{\mathfrak{Im}\;}

\newcommand{\spn}{{\mathrm{span}\,}}

\newcommand{\vc}{\boldsymbol}

\newcommand{\R}{\mathbb{R}}
\newcommand{\C}{\mathbb{C}}

\newcommand{\dd}{\,\text{d}}
\newcommand{\supp}{\mbox{supp\;}}

\newtheorem{thm}{Theorem}
\newtheorem{prop}{Proposition}
\newtheorem{lem}{Lemma}

\title{A density property for tensor products of gradients of harmonic functions and applications}
\date{}

\author{C\u{a}t\u{a}lin I. C\^{a}rstea\thanks{School of Mathematics, Sichuan University, Chengdu, Sichuan, 610064, P.R.China; email: catalin.carstea@gmail.com} \and Ali Feizmohammadi\thanks{Department of Mathematics, University College London, London, UK-WC1E 6BT, United Kingdom; email: a.feizmohammadi@ucl.ac.uk}}\date{}

\begin{document}
\maketitle

\begin{abstract}
We show that tensor products of $k$ gradients of harmonic functions, with $k$ at least three, are dense in $C(\overline{\Omega})$, for any bounded domain $\Omega$ in dimension 3 or higher. The bulk of the argument consists in showing that any smooth compactly supported $k$-tensor that is $L^2$-orthogonal to all such products must be zero. This is done by using a Gaussian quasi-mode based construction of harmonic functions in the orthogonality relation. We then demonstrate the usefulness of this result by using it to prove uniqueness in the inverse boundary value problem for a coupled quasilinear elliptic system. The paper ends with a discussion of the corresponding property for products of two gradients of harmonic functions, and the connection of this property with the linearized anisotropic Calder\'on problem.
\end{abstract}

\section{Introduction}

Let $\dom\subset\R^{1+n}$, $n\geq 2$ be a bounded domain. (We will label coordinates from $0$ to $n$.) It has already been known for some time that the set
\begin{equation}
\spn\{u_1u_2:u_1,u_2\text{ are harmonic in a neighborhood of }\overline{\dom}\}
\end{equation}
is dense in $C(\overline{\Omega})$. For example, we may follow \cite[Chapter 5]{isakov-book} in pointing out that if this were not the case, it is a consequence of the Hahn-Banach theorem that there exists a non-zero measure $\mu$ with support in $\overline{\dom}$ such that
\begin{equation}
\int u_1(\vc x)u_2(\vc x)\dd\mu(\vc x)=0,\quad\forall u_1,u_2\text{ harmonic in a neighborhood of }\overline{\dom}.
\end{equation}
In particular, for any $\vc z\not\in\overline{\dom}$ we can take $u_1(\vc x)=u_2(\vc x)=|\vc z-\vc x|^{2-n}$, which are harmonic. Then
\begin{equation}
\int |\vc z-\vc x|^{4-2n}\dd\mu(\vc x)=0,\quad\forall \vc z\not\in\supp(\mu),
\end{equation}
which implies, for example by a result of Riesz in \cite[Chapitre III]{riesz}, that $\mu=0$. 

We have quoted the proof above because it shows that in some way this can be seen as a very old result. Another classic way of proving the same thing is due to Calder\'on. In his paper \cite{Ca} he introduced the following harmonic functions. Let $\vc\xi,\vc\nu\in\R^n\setminus\{0\}$ be such that $\vc\xi\perp\vc\nu$, $|\vc\nu|=1$,   and let
\begin{equation}
\vc\zeta_{\pm}=\frac{1}{2}\left(\vc\xi\pm i|\vc\xi|\vc\nu\right).
\end{equation}
It is easy to check that
\begin{equation}
\vc\zeta_{+}\cdot\vc\zeta_{+}=\vc\zeta_{-}\cdot\vc\zeta_{-}=0,\quad \vc\zeta_{+}+\vc\zeta_{-}=\vc\xi.
\end{equation}
With these choices $e^{i\vc\zeta_{\pm}\cdot\vc x}$ is a harmonic function. 
If we choose 
\begin{equation}
u_1(\vc x)=e^{i\vc\zeta_+\cdot\vc x},\quad u_2(\vc x)=e^{i\vc\zeta_-\cdot\vc x},
\end{equation}
and $\mu$ is a measure as above, then
\begin{equation}
0=\int u_1(\vc x)u_2(\vc x)\dd\mu(\vc x)=\hat\mu(\vc \xi),\quad\forall \vc\xi\in\R^n\setminus\{0\}.
\end{equation}
Again it follows that $\mu=0$. We will use these harmonic functions again later in this paper.

Note also that we can  trivially conclude that
\begin{equation}
\spn\{u_1u_2\ldots u_k:u_1,\ldots,u_k\text{ are harmonic in a neighborhood of }\overline{\dom}\}
\end{equation}
is dense in $C(\overline{\Omega})$, for any $k\geq 2$, since $u_3$, \ldots, $u_k$ can be chosen to be constant.

There are various results generalizing this density result from harmonic functions to solutions of other partial differential equations. Notably, Cal\-de\-r\'on's construction has been generalized in \cite{SU} to solutions of the equation $\Delta u-qu=0$, which they then  use to solve Calder\'on's inverse boundary value problem, proposed in \cite{Ca}. In fact, examples of such density results are too numerous to comprehensively list here. We mention \cite{isakov7} for equations in a general form, and \cite{fksu}, \cite{buu} for solutions to elliptic equations that are zero on part of  the boundary of the domain.

\subsection{Our result}

In this paper we wish to consider the analogous problem for products of gradients of harmonic functions. We will show the following.

\begin{thm}\label{main-thm}
Let $\mathcal{O}$ be a bounded neighborhood of $\dom$. The set
\begin{equation}
\spn\{\nabla u_1\otimes\nabla u_2\otimes\cdots\otimes\nabla u_k: u_1, u_2,\ldots, u_k\text{ are harmonic in $\mathcal{O}$}\}
\end{equation}
with $k\geq3$ is dense in $C(\overline{\dom},\C^{\otimes k})$.
\end{thm}

As above, it is sufficient to take $k=3$ since the $u_4$, \ldots, $u_k$ can be chosen from among the coordinate functions $\vc x\to x_j$, $j=0,\ldots,n$. Also as above, by the Hahn-Banach theorem, the following statement is sufficient to prove Theorem \ref{main-thm}.

\begin{thm}\label{integral-thm}
Suppose $\mu_{jkl}$ are finite measures with support in $\overline\dom$ such that
\begin{equation}
\sum_{j,k,l=0}^n\int\pr_ju_1(\vc x)\pr_ku_2(\vc x)\pr_lu_3(\vc x)\dd\mu_{jkl}(\vc x)=0,
\end{equation}
for all $u_1$, $u_2$, $u_3$ harmonic functions in $\mathcal{O}$. Then $\mu_{jkl}=0$, $j,k,l=0,\ldots,n$.
\end{thm}

This result generalizes the one obtained in \cite{carfeiz} where it is assumed that $\mu_{jkl}=\mu_{kjl}$, for all $j,k,l=0.\ldots,n$. The main gain of this paper is that this symmetry requirement is shown to not be necessary.

The proof of this Theorem \ref{integral-thm} can be reduced to a similar result for measures of the form $B_{jkl}(\vc x)\dd\vc x$, with smooth, compactly supported $B_{jkl}$. To see that this is the case, first notice that for small enough $y$ and for $\mu_{jkl}$ as in the statement of the theorem,
\begin{equation}
\sum_{j,k,l=0}^n\int \pr_ju_1(\vc y+\vc x)\pr_ku_2(\vc y+\vc x)\pr_lu_3(\vc y+\vc x)\dd\mu_{jkl}(\vc x)=0,
\end{equation}
for all $u_1$, $u_2$, $u_3$ harmonic functions in a fixed neighborhood of $\overline{\dom}$. This is clearly true since translations of harmonic functions are still harmonic. 
Now let $\varphi_\epsilon$ be a smooth, compactly supported, approximation of identity. Then
\begin{multline}
\sum_{j,k,l=0}^n\int (\varphi_\epsilon*\mu_{jkl})(\vc x)\pr_ju_1(\vc x)\pr_ku_2(\vc x)\pr_lu_3(\vc x)\dd\vc x\\[5pt]
=\sum_{j,k,l=0}^n\int \varphi_\epsilon(\vc y)
\left(\int\pr_ju_1(\vc y+\vc x)\pr_ku_2(\vc y+\vc x)\pr_lu_3(\vc y+\vc x)\dd\mu_{jkl}(\vc x)\right)\dd\vc y\\[5pt]=0.
\end{multline}
It follows that it is enough to prove the following proposition.

\begin{prop}\label{prop}
Suppose $B=(B_{jkl})$  is a 3-tensor such that $B\in C_0^\infty(\mathcal{O})$,  and
\begin{equation}\label{basic-integral-identity}
\int B(\vc x):\nabla u_1(\vc x)\otimes\nabla u_2(\vc x)\otimes\nabla u_3(\vc x)\dd \vc x=0
\end{equation}
for all smooth functions $u_1$, $u_2$, $u_3$ which are harmonic in $\mathcal{O}$. Then $B=0$.
\end{prop}

We can further simplify our task by appealing to a result proved in \cite{carfeiz}. Suppose $B$ is as in the statement of Proposition \ref{prop} and let 
\begin{equation}
C_{jkl}=\frac{1}{2}(B_{jkl}+B_{kjl}).
\end{equation}
Then for any harmonic functions $u_1$, $u_2$, $u_3$ we have
\begin{multline}
\int C(\vc x):\nabla u_1(\vc x)\otimes\nabla u_2(\vc x)\otimes\nabla u_3(\vc x)\dd\vc x\\[5pt]
=\frac{1}{2}\int B(\vc x):\nabla u_1(\vc x)\otimes\nabla u_2(\vc x)\otimes\nabla u_3(\vc x)\dd\vc x\\[5pt]+\frac{1}{2}\int B:\nabla u_1(\vc x)\otimes\nabla u_3(\vc x)\otimes\nabla u_2(\vc x)\dd\vc x=0.
\end{multline}
It was shown in \cite{carfeiz} that $C=0$. It follows that  
\begin{equation}
B_{jkl}(\vc x)=-B_{jlk}(\vc x),\quad\forall\vc x\in\mathcal{O},
\end{equation}
which we will assume to be the case for the rest of this paper.

The harmonic functions introduced by Calder\'on are not sufficient for the proof of our result. We will instead work with a family of harmonic functions whose construction is based on Gaussian quasi-modes on arbitrary hyperplanes in $\R^{1+n}$. These are approximate eigenfunctions to the Laplacian on $\R^n$ which concentrate along straight lines. Using these, we can then construct harmonic functions, defined on bounded subsets of $\R^{1+n}$, that concentrate on planes. Such constructions  have been  used recently to solve inverse problems for linear elliptic equations in \cite{KSa}, \cite{DKLS} and non-linear elliptic equations in \cite{FO}, \cite{LLLS1}.

The construction of Gaussian quasi-modes in elliptic equations is based on an analogue for hyperbolic equations, namely Gaussian beams. These are approximate solutions to the wave equation that concentrate on null geodesics. They were introduced in the works \cite{BU}, \cite{Ralston} and have been used in the context of inverse problems in many works. For example see \cite{KKL} and the references therein. 
Solutions to equations of the form $\Delta u+qu=0$ that concentrate on planes have also been constructed in \cite{GU} by a different method.

In section \ref{gauss-construction} we give a summary of the construction of Gaussian quasi-mode harmonic functions. The details of the particular incarnation of the construction we are using have been worked out in \cite{carfeiz}, so we only provide the results of the computations here. The harmonic functions we introduce depend on an assymptotic parameter $\lambda$ which will later be made to go to infinity. There are assymptotic expansions (in $\lambda^{-1}$ and $ x_2$) for these harmonic functions and we give a few of the first terms whose exact expressions we will need to use. It is also important for our later computation to note that our harmonic functions also depend on a number of parameters that can be arbitrarily chosen.

In section \ref{stationary} we use the special harmonic functions constructed in the previous section in order to obtain information form \eqref{basic-integral-identity}. With appropriate choices for $u_1$, $u_2$, $u_3$ , we let $\lambda\to\infty$. Using a stationary phase theorem we obtain an expansion of the left hand side into powers of $\lambda^{-1}$. The coefficients of each power of $\lambda^{-1}$ must each be zero independently, so they each provide possibly distinct information.

In section \ref{expansion} we carry out the details of the above mentioned expansion. It turns out to be necessary to use both the first and second orders in the expansion given by the stationary phase theorem (the zero order is automatically zero in our case, so provides no information). We obtain enough information about the structure of $B$ that at the last step we are able to use Calder\'on harmonic functions in order to conclude that it must be zero.

In section \ref{application} we give an application of our result (of Theorem \ref{integral-thm} to be precise). This is a proof of uniqueness in the inverse boundary value problem for a coupled quasilinear system of two equations of the form
\begin{equation}
\left\{\begin{array}{l} -\Delta u^J+\sum_{K=1}^2\nabla(A^{JK}:\nabla u^J\otimes\nabla u^K)=0,\quad J=1,2,\\[5pt] (u^1,u^2)|_{\pr\dom}=(f^1,f^2).\end{array}\right.
\end{equation}
The method we use for reducing the question of uniqueness to the statement of Theorem \ref{integral-thm} is the standard ``second linearization'' approach, first used in \cite{I1}.
For elliptic semilinear or quasilinear equations, there are a number of uniqueness results that are known. For semilinear equations, examples include \cite{FO}, \cite{IN}, \cite{IS}, \cite{KU}, \cite{KU2},   \cite{LLLS1}, \cite{LLLS2}, \cite{S2}. For quasilinear equations, not in divergence form, see \cite{I2}. For  quasilinear equations in divergence form see  \cite{CK}, \cite{CNV},  \cite{EPS}, \cite{HS}, \cite{KN}, \cite{MU},  \cite{Sh}, \cite{S1}, \cite{S3}, \cite{SuU}, (also \cite{C} for quasilinear time-harmonic Maxwell systems).

Finally, it is natural to ask what happens when we take products of only two gradients of harmonic functions. Indeed, this question appears naturally when one considers the linearized inverse boundary value problem for an elliptic equation with anisotropic coefficients (or the linearized anisotropic Calder\'on problem). In section \ref{calderon} we will explain this connection and prove the following theorem.
\begin{thm}\label{lin-cal}
Suppose $C$ is a  matrix with coefficients $C_{jk}\in L^p(\dom)$, $j,k=0,\ldots,n$, $1<p<\infty$. Then 
\begin{equation}\label{2-int-id}
\int_\dom C(\vc x):\nabla u_1(\vc x)\otimes\nabla u_2(\vc x)\dd\vc x=0,
\end{equation}
for all $u_1$, $u_2$ harmonic functions in $\R^{1+n}$ {\bf if and only if} there exist $v_j\in W^{1,p}_0(\dom)$, $a_{j,k}\in L^p(\dom)$,  with $a_{jk}=-a_{kj}$,  $\sum_j\pr_j a_{jk}=0$, $j,k=0,\ldots,n$, such that
\begin{equation}
C_{jk}=\pr_j v_k+\pr_k v_j-\delta_{jk}\nabla\cdot\vc v+a_{jk}.
\end{equation}
The requirement that $\sum_j\pr_j a_{jk}=0$ is to be taken in the sense of $\mathcal{E}'(\R^{1+n})$.
\end{thm}

This result is the anisotropic analogue of \cite{Ca}. The methods we use to prove Theorem \ref{lin-cal}, unlike those employed in the proof of Theorem \ref{integral-thm}, are not particularly sophisticated. In fact, we only need to use harmonic functions of the form $e^{i\vc\zeta_\pm\cdot\vc x}$ in \eqref{2-int-id}, together with the same methods for improving the regularity of $C$ used above, in order to obtain the result. It therefore seems like some version of Theorem \ref{lin-cal} should have been known a long time ago. A version of the result appears in \cite[Theorem 1.13]{shar}, for  $C$ smooth\footnote{We thank Mikko Salo for pointing out this reference.}. We provide our own proof below for the benefit of the interested reader.

\section{Gaussian quasi-mode construction of harmonic functions}\label{gauss-construction}

In this section we will give a brief description of a family of harmonic functions constructed using Gaussian quasi-modes on hyperplanes in $\R^{1+n}$. This is a summary of results obtained in \cite{carfeiz} and the details of the computations are left out.

For the coordinates of $\vc x\in\R^{1+n}$ we will use the notation
\begin{equation}
\vc x=( x_0, x_1, x_2,\ldots, x_n)=( x_0,\vc x')=( x_0, x_1,\vc x'')=( x_0, x_1, x_2,\vc x'''),
\end{equation}
and correspondingly introduce the differential operators
\begin{equation}
\nabla=(\pr_0,\pr_1,\ldots,\pr_n),\quad \nabla'=(0,\pr_1,\ldots,\pr_n)
\end{equation}
and
\begin{equation}
\Delta=\sum_{j=0}^n\pr_j^2,\quad \Delta'=\sum_{j=1}^n\pr_j^2.
\end{equation}
Operators $\nabla''$, $\nabla'''$, $\Delta''$, $\Delta'''$ can also be defined in the same way.
Note that without any loss of generality we may assume that $\mathcal{O}\subset\{x_1>0\}$.

Let $\tau=\lambda+i\sigma$, $\lambda, \sigma\in\R$. We construct harmonic functions of the form
\begin{equation}\label{quasi-construction}
u_\tau^\pm(\vc x)=e^{\pm\lambda x_0}\left( e^{\pm i\sigma x_0} e^{i\tau\Psi(x_1,x_2)}a_\tau(x_1,x_2)+r_\tau(\vc x)   \right).
\end{equation}
Note that
\begin{equation}
\Delta u_\tau^\pm=e^{\pm\tau x_0}[\tau^2+\Delta'] \left(e^{i\tau\Psi}a_\tau\right)
+\Delta \left(e^{\pm\lambda x_0}r_\tau\right),
\end{equation}
and
\begin{multline}\label{quasi-op}
(\tau^2+\Delta') e^{i\tau\Psi}a_\tau\\[5pt]=
e^{i\tau\Psi}\left[ \tau^2(1-|\nabla'\Psi|^2)a_\tau+
i\tau(2\nabla'\Psi\cdot\nabla' a_\tau+(\Delta'\psi)a_\tau)
+(\Delta' a_\tau) \right].
\end{multline}

The quantity $\lambda$ will be an asymptotic parameter, in the sense that we eventually intend to take the limit $\lambda\to\infty$. We will construct $\Psi$ and $a_\tau$ so that this quantity vanishes to high order in both $\lambda^{-1}$ and $x_2$.

Let $M$ be a large natural number, $\delta>0$, let $\chi:\R\to[0,\infty)$ be a smooth function such that $\chi(t)=1$ for $|t|<\frac{1}{2}$ and $\chi(t)=0$ for $|t|>1$, and let $h:\R^{n-1}\to\R$ be a harmonic function.
We make the following Ans\"atze:
\begin{equation}\label{An-1}
\Psi(x_1,x_2)=\sum_{j=0}^M\psi_j(x_1)x_2^j,
\end{equation}
\begin{equation}\label{An-2}
a_\tau(x_1,x_2)=
\chi(\frac{x_2}{\delta})h(x''')
\sum_{k=0}^Mv_k(x_1,x_2)\tau^{-k},
\end{equation}
\begin{equation}\label{An-3}
v_k(x_1,x_2)=\sum_{j=0}^M v_{k;j}(x_1)x_2^j.
\end{equation}

Since we would like for the quantity in equation \eqref{quasi-op} to vanish to high order in $\lambda^{-1}$ and $x_2$, we require that the  following conditions hold:
{\setlength{\baselineskip}{1.5\baselineskip}\begin{itemize}
\item[(i)]$\im \Psi\geq \kappa |x_2|^2$ ;
\item[(ii)]$\pr_2^j(|\nabla'\Psi|^2-1)|_{x_2=0}=0$ for $j=0,1,\ldots,M$;
\item[(iii)]$\pr_2^j(2\nabla'\Psi\cdot\nabla' v_0+(\Delta'\Psi)v_0)|_{x_2=0} =0$ for $j=0,1,\ldots,M$.
\item[(iv)]$\pr_2^j(2\nabla'\Psi\cdot\nabla' v_k+(\Delta'\Psi)v_k-i\Delta'v_{k-1})|_{x_2=0} =0$ for $j=0,1,\ldots,M$.
\end{itemize}
}
This conditions do not uniquely determine the $\psi_j$, $v_{k;j}$. They do however provide ODEs that these quantities need to satisfy. This provides a way to compute the $\psi_j$, $v_{k;j}$, at least up to some choices and arbitrary constants. The first few of these terms, as computed in \cite{carfeiz}, are:
\begin{equation}
\psi_0(x_1)=x_1,\quad \psi_1(x_1)=0
\end{equation}
\begin{equation}
\psi_2(x_1)=\frac{1}{2}(x_1-i\epsilon)^{-1}
\end{equation}
\begin{equation}
\psi_3=p_3(x_1-i\epsilon)^{-3},\quad p_3\in\C
\end{equation}
\begin{equation}
\psi_4(x_1)=-\frac{1}{8}(x_1-i\epsilon)^{-3}+\frac{9}{2}p_3^2(x_1-i\epsilon)^{-5}+p_4(x_1-i\epsilon)^{-4},\quad p_4\in\C.
\end{equation}
\begin{multline}
\psi_5(x_1)
=27p_3^3(x_1-i\epsilon)^{-7}+12p_3p_4(x_1-i\epsilon)^{-6}
\\[5pt]
 +p_5(x_1-i\epsilon)^{-5},\quad p_5\in\C.
\end{multline}
\begin{equation}
v_{0;0}=(x_1-i\epsilon)^{-\frac{1}{2}}.
\end{equation}
\begin{equation}
v_{0;1}=3p_3(x_1-i\epsilon)^{-\frac{5}{2}}+q_1(x_1-i\epsilon)^{-\frac{3}{2}},\quad q_1\in\C.
\end{equation}
Furthermore, the remainder terms can be chosen such that
\begin{equation}
||r_\tau||_{H^{k}(\mathcal{O})}=\mathscr{O}(|\tau|^{-\frac{M-1-3k}{2}}).
\end{equation}
This remainder estimate was obtained in \cite{carfeiz} with the help of results from \cite{FO}.

\section{Stationary phase as $\lambda\to\infty$}\label{stationary}

In this section we will plug the solutions of the previous section into \eqref{basic-integral-identity}, then use a stationary phase theorem in order to identify the different orders in $\lambda$ of the left hand side.

Let $\vc e_j$ be the $j^{th}$ coordinate unit vector. Let $\vc\alpha=\vc e_0+i\vc e_1$, and let the quantities $V_\tau^\pm$ be 
\begin{multline}
V_\tau^+=\tau^{-1} e^{-\tau x_0}e^{-i\tau\Psi}\nabla u_\tau^+\\[5pt]
=\chi h\Big[v_{0;0}\vc\alpha+x_2\left[v_{0;1}\vc\alpha+2iv_{0;0}\psi_2\vc e_2  \right]\\[5pt]
+x_2^2\left[v_{0;2}\vc\alpha+iv_{0;0}\dot\psi_2\vc e_1+i(2v_{0;1}\psi_2+3v_{0;0}\psi_3)\vc e_2  \right]\\[5pt]
+\tau^{-1}\left[v_{1;0}\vc\alpha +\dot v_{0;0}\vc e_1+v_{0;1}\vc e_2 \right]\\[5pt]
+x_2^3\left[v_{0;3}\vc\alpha+i(v_{0;1}\dot\psi_2+v_{0;0}\dot\psi_3)\vc e_1 +i(2v_{0;2}\psi_2+3v_{0;1}\psi_3+4v_{0;0}\psi_4)\vc e_2   \right]\\[5pt]
+\tau^{-1}x_2\left[ v_{1;1}\vc\alpha+\dot v_{0;1}\vc e_1 +(2iv_{1;0}\psi_2+2v_{0;2})\vc e_2 \right]\Big]\\[5pt]
+\tau^{-1}\left[ \delta^{-1}\dot\chi h\vc e_2+\chi\nabla''' h\right](v_{0;0}+v_{0;1}x_2)+\cdots,
\end{multline}
\begin{multline}
(-1)\overline{V_\tau^-}:=(-1)\overline{\tau^{-1} e^{\tau x_0}e^{i\tau\Psi}\nabla u_\tau^-}\\[5pt]
=\chi h\Big[\overline{v_{0;0}}\vc\alpha+x_2\left[\overline{v_{0;1}}\vc\alpha+2i\overline{v_{0;0}}\overline{\psi_2}\vc e_2  \right]\\[5pt]
+x_2^2\left[\overline{v_{0;2}}\vc\alpha+i\overline{v_{0;0}}\overline{\dot\psi_2}\vc e_1+i(2\overline{v_{0;1}}\overline{\psi_2}+3\overline{v_{0;0}}\overline{\psi_3})\vc e_2  \right]\\[5pt]
+\overline{\tau^{-1}}\left[\overline{v_{1;0}}\vc\alpha -\overline{\dot v_{0;0}}\vc e_1-\overline{v_{0;1}}\vc e_2 \right]\\[5pt]
+x_2^3\left[\overline{v_{0;3}}\vc\alpha+i(\overline{v_{0;1}}\overline{\dot\psi_2}+\overline{v_{0;0}}\overline{\dot\psi_3})\vc e_1 +i(2\overline{v_{0;2}}\overline{\psi_2}+3\overline{v_{0;1}}\overline{\psi_3}+4\overline{v_{0;0}}\overline{\psi_4})\vc e_2   \right]\\[5pt]
+\overline{\tau^{-1}}x_2\left[ \overline{v_{1;1}}\vc\alpha-\overline{\dot v_{0;1}}\vc e_1 +(2i\overline{v_{1;0}}\overline{\psi_2}-2\overline{v_{0;2}})\vc e_2 \right]\Big]\\[5pt]
-\overline{\tau^{-1}}\left[ \delta^{-1}\dot\chi h\vc e_2+\chi\nabla''' h\right](\overline{v_{0;0}}+\overline{v_{0;1}}x_2)+\cdots.
\end{multline}

In \eqref{basic-integral-identity} we choose
\begin{equation}
u_1=u_{1;\tau}^+,\quad u_2=u_{2;\tau}^+,\quad u_3=\overline{u_{3;2\tau}^-},
\end{equation}
where the numerical indices indicate that we might have different choices in the constants $p_j$,etc., and different choices of harmonic functions $h_1$, $h_2$, and $h_3$. For the various constants that appear in the expansions of our special solutions we will use a superscript to indicate the solution to which it belongs. For example $q_1^2$ will be the $q_1$ constant associated to solution 2.

For convenience, we choose $p_3^j=p_4^j=p_5^j=0$ for $j=1,2,3$. This will simplify somewhat the expressions  that  follow below.

With these choices
\begin{equation}
e^{\tau x_0}e^{i\tau\Psi_1}e^{\tau x_0}e^{i\tau\Psi_2}e^{-2\overline{\tau} x_0}e^{-2i\overline{\tau}\overline{\Psi_3}}
=e^{4i\sigma x_0}e^{i(\tau\Psi_1+\tau\Psi_2-2\overline{\tau}\overline{\Psi_3})},
\end{equation}
where we can expand
\begin{equation}
\tau\Psi_1+\tau\Psi_2-2\overline{\tau}\overline{\Psi_3}
=4i\sigma x_1+\sigma F_1+\lambda F_2,
\end{equation}
with
\begin{equation}
F_1=\frac{2ix_1}{x_1^2+\epsilon^2}x_2^2-\frac{i}{2}\frac{x_1^3-3x_1\epsilon^2}{(x_1^2+\epsilon^2)^3}x_2^4+\mathscr{O}(x_2^6),
\end{equation}
\begin{equation}
F_2=\frac{2i\epsilon}{x_1^2+\epsilon^2}x_2^2+\frac{i}{2}\frac{\epsilon^3-3x_1^2\epsilon}{(x_1^2+\epsilon^2)^3}x_2^4+\mathscr{O}(x_2^6).
\end{equation}

We have
\begin{multline}
\tau^{-2}\overline{\tau}^{-1}\frac{1}{2}B : \nabla u_1\otimes\nabla u_2\otimes\nabla u_3
\\[5pt]
=e^{i\lambda F_2}e^{4i\sigma x_0}e^{-4\sigma x_1}e^{i\sigma F_1}B :  V_{1;\tau}^+\otimes V_{2;\tau}^+\otimes\overline{V_{3;2\tau}^-}.
\end{multline}
We  quote here a stationary phase theorem, slightly addapted from the source to fit our particular circumstances.
\begin{thm}[{see \cite[Theorem 7.7.5]{H1}}]\label{H-thm}
Let $0\in I\subset \R$ be a bounded interval. Let $X$ be an open neighborhood of $\overline{I}$. If $U\in C_{0}^{2k}(I)$, $F\in C^{3k+1}(X)$ and $\im F\geq0$ in $X$, $F'(0)=0$, $ F''(0)\neq 0$, $F'\neq0$ in $I\setminus\{0\}$. Let
\begin{equation}
G(t)=F(t)-F(0)-\frac{1}{2}F''(0)t^2
\end{equation}
and
\begin{equation}
L_jU=\sum_{\nu-\mu=j}\sum_{2\nu\geq3\mu}i^{-j}\frac{1}{\nu!\mu!}\left(-\frac{1}{2F''(0)} \right)^\nu\frac{\dd^{2\nu}}{\dd t^{2\nu}}\left(G^\mu U\right)(0).  
\end{equation}
Then
\begin{equation}
\left|\int_I e^{i\lambda F(t)}U(t)-e^{i\lambda F(0)}\left( \frac{2\pi i}{\lambda F''(0)} \right)^{\frac{1}{2}}\sum_{j<k}\lambda^{-j}L_jU   \right|\leq C\lambda^{-k}||U||_{C^{2k}(I)}.
\end{equation}
\end{thm}

With the following definitions, this theorem can be applied to the integral in the $x_2$ variable part of \eqref{basic-integral-identity}. We define
\begin{equation}
U=(-1)e^{i\sigma F_1}B :  V_{1;\tau}^+\otimes V_{2;\tau}^+\otimes\overline{V_{3;2\tau}^-},\quad F(x_2)=F_2(x_2),
\end{equation}
which then gives
\begin{equation}
G(x_2)=F_2(x_2)-F_2(0)-\frac{1}{2}F_2''(0)x_2^2=\frac{i}{2}\frac{\epsilon^3-3x_1^2\epsilon}{(x_1^2+\epsilon^2)^3}x_2^4+\mathscr{O}(x_2^6).
\end{equation}
Here we treat $F_2$ as a function of $x_2$.
With this $G$ we have
\begin{equation}
L_0(U)=U|_{x_2=0},
\end{equation}
\begin{equation}
L_1(U)=\frac{1}{8\epsilon}\left[(x_1^2+\epsilon^2)\pr_2^2U+\frac{1}{8}\frac{3x_1^2-\epsilon^2}{x_1^2+\epsilon^2}U    \right]_{x_2=0},
\end{equation}
and
\begin{equation}
L_2(U)=\frac{1}{128\epsilon^2}\left[(x_1^2+\epsilon)^2\pr_2^4U
+\frac{15(3x_1^2-\epsilon^2)}{2}\pr_2^2U
+\frac{105(3x_1^2-\epsilon^2)^2}{16(x_1^2+\epsilon^2)^2}U  \right]_{x_2=0}.
\end{equation}
Theorem \ref{H-thm} gives  that
\begin{multline}\label{hormander}
\tau^{-2}\overline{\tau}^{-1}\frac{(-1)}{2}\int B : \nabla u_1\otimes\nabla u_2\otimes\nabla u_3\dd x_2\\[5pt]
=\left(\pi \frac{x_1^2+\epsilon^2}{2\lambda\epsilon} \right)^{\frac{1}{2}}e^{4i\sigma x_0}e^{-4\sigma x_1}\left[L_0(U)+\lambda^{-1}L_1(U) +\lambda^{-2}L_2(U) \right]\\[5pt]+\mathscr{O}(\lambda^{-\frac{7}{2}}).
\end{multline}

\section{Expansion of \eqref{hormander}}\label{expansion}

Here we will collect the first few terms in the expansion in $\lambda^{-1}$ of \eqref{hormander}, or at least parts of these terms that can be isolated  by independently varying the arbitrary parameters $q_1^j$. We do it by computing coefficients in the expansion in $\lambda^{-1}$ and $x_2$ of $-e^{-i\sigma F_1}U=B :  V_{1;\tau}^+\otimes V_{2;\tau}^+\otimes\overline{V_{3;2\tau}^-}$. Clearly the $\lambda^0x_2^0$ term of $B :  V_{1;\tau}^+\otimes V_{2;\tau}^+\otimes\overline{V_{3;2\tau}^-}$, which is 
\begin{equation}
\chi^3h_1h_2h_3v_{0;0}|v_{0;0}|^2 B:\vc\alpha\otimes\vc\alpha\otimes\vc\alpha,
\end{equation}
 vanishes by symmetry. 

\subsection{The $\lambda^{-\frac{3}{2}}q_1^j$ term in \eqref{hormander}}

The $\lambda^{-1}x_2^0$ and $\lambda^0 x_2^2$ terms in $B :  V_{1;\tau}^+\otimes V_{2;\tau}^+\otimes\overline{V_{3;2\tau}^-}$ contribute to this. In fact, we will only keep the terms that also contain a factor of $q_1^j$, $j=1,2,3$. Since these constants can be varied independently, the expression we will end up with will also vanish.

The coefficient of $\lambda^{-1}x_2^0$ is
\begin{multline}
\chi^3h_1h_2h_3 B:\Big(
\dot v_{0;0}|v_{0;0}|^2\vc e_1\otimes\vc\alpha\otimes\vc\alpha
+v_{0;1}^1|v_{0;0}|^2 \vc e_2\otimes\vc\alpha\otimes\vc\alpha\\[5pt]
+\dot v_{0;0}|v_{0;0}|^2\vc\alpha\otimes\vc e_1\otimes\vc\alpha
+v_{0;1}^2|v_{0;0}|^2\vc\alpha\otimes\vc e_2\otimes\vc\alpha
\Big),
\end{multline}
with other  terms being zero by reason of symmetry.

The the coefficient of $\lambda^0x_2^2$ contains terms that can be obtained in two different ways, which we label ``$x_2^2\times 1\times 1$'' and ``$x_2\times x_2\times 1$''.

 The ``$x_2^2\times 1\times 1$''terms are
\begin{multline}
\chi^3h_1h_2h_3 B:\Big(
iv_{0;0}|v_{0;0}|^2\dot\psi_2\vc e_1\otimes\vc\alpha\otimes\vc\alpha
+2i v_{0;1}^1\psi_2|v_{0;0}|^2\vc e_2\otimes\vc\alpha\otimes\vc\alpha\\[5pt]
+iv_{0;0}|v_{0;0}|^2\dot\psi_2\vc\alpha\otimes\vc e_1\otimes\vc\alpha
+2i v_{0;1}^2\psi_2|v_{0;0}|^2\vc\alpha\otimes\vc e_2\otimes\vc\alpha
\Big)\\[5pt]
=\chi^3h_1h_2h_3 B:\Big(
2i v_{0;1}^1\psi_2|v_{0;0}|^2\vc e_2\otimes\vc\alpha\otimes\vc\alpha
+2i v_{0;1}^2\psi_2|v_{0;0}|^2\vc\alpha\otimes\vc e_2\otimes\vc\alpha
\Big).
\end{multline}

The ``$x_2\times x_2\times 1$'' terms are
\begin{multline}
\chi^3h_1h_2h_3 B:\Big(
2i\psi_2|v_{0;0}|^2v_{0;1}^1\vc\alpha\otimes\vc e_2\otimes\vc\alpha
+2i\psi_2|v_{0;0}|^2 v_{0;1}^2 \vc e_2\otimes\vc\alpha\otimes\vc\alpha\\[5pt]
+2i \psi_2(v_{0;0})^2\overline{v_{0;1}^3}\vc e_2\otimes\vc\alpha\otimes\vc\alpha
+2i \psi_2(v_{0;0})^2\overline{v_{0;1}^3}\vc\alpha\otimes\vc e_2\otimes\vc\alpha
\Big)\\[5pt]
=\chi^3h_1h_2h_3 B:\Big(
2i\psi_2|v_{0;0}|^2v_{0;1}^1\vc\alpha\otimes\vc e_2\otimes\vc\alpha
+2i\psi_2|v_{0;0}|^2 v_{0;1}^2 \vc e_2\otimes\vc\alpha\otimes\vc\alpha
\Big).
\end{multline}

Adding these two we see that the coefficient of $\lambda^0x_2^2$ in $B :  V_{1;\tau}^+\otimes V_{2;\tau}^+\otimes\overline{V_{3;2\tau}^-}$ is zero.

In order to compute the $\lambda^0$ order term in $L_1(U)$ we also need to account for two terms containing derivatives of $B$. The first one is
\begin{multline}
\left.\pr_2^2(e^{i\sigma F_1}B) : V_{1;\tau}\otimes V_{2;\tau}\otimes\overline{V}_{3;-2\tau}\right|_{ x_2=0}\\[5pt]
=v_{0;0}|v_{0;0}|^2\left.h_1h_2h_3\pr_2^2(e^{i\sigma F_1}B):\vc\alpha\otimes\vc\alpha\otimes\vc\alpha\right|_{ x_2=0}+\mathscr{O}(\lambda^{-1})\\[5pt]=\mathscr{O}(\lambda^{-1}).
\end{multline}
The  second one is
\begin{multline}
\left.\pr_2(e^{i\sigma F_1}B) : \pr_2(V_{1;\tau}\otimes V_{2;\tau}\otimes\overline{V}_{3;-2\tau})\right|_{ x_2=0}
=\left.h_1h_2h_3\pr_2(e^{i\sigma F_1}B)\right|_{ x_2=0}\\[5pt]:\Big[
(v_{0;1}^1|v_{0;0}|^2+v_{0;1}^2|v_{0;0}|^2+\overline{v_{0;1}^3}(v_{0;0})^2)\vc\alpha\otimes\vc\alpha\otimes\vc\alpha\\[5pt]
+2i v_{0;0}|v_{0;0}|^2(\psi_2\vc e_2\otimes\vc\alpha\otimes\vc\alpha+
\psi_2\vc\alpha\otimes\vc e_2\otimes\vc\alpha+
\overline{\psi}_2\vc\alpha\otimes\vc\alpha\otimes\vc e_2)
\Big]\\[5pt]+\mathscr{O}(\lambda^{-1})=\mathscr{O}(\lambda^{-1}).
\end{multline}
Both  contribute nothing to the computation.

Finally, the $\lambda^{-1}$ term in  $L_0(U)$ may contain terms involving $[\delta^{-1}\dot\chi h_j \vc e_2+\chi\nabla''' h_j]$. However, none of these will contain a factor $q_1^j$, so they contribute nothing to the computation.

From the $\lambda^{-\frac{3}{2}}q_1^1$ term in \eqref{hormander} we then get
\begin{multline}
0
=\int e^{-4\sigma x_1}\hat B(4\sigma,x_1,0,\vc x''') : \vc e_2\otimes\vc\alpha\otimes\vc\alpha \\[5pt] \times (x_1-i\epsilon)^{-\frac{3}{2}}h_1(\vc x''')h_2(\vc x''')h_3(\vc x''')\dd x_1\dd\vc x'''.
\end{multline}
Here, by $\hat B$ we mean the Fourier transform of $B$ only in the $x_0$ variable.
 
As we have mentioned at the beginning of this paper, the span of products of at least two harmonic functions is dense. This allows us to remove the $\dd\vc x'''$ integral to obtain that
\begin{equation}
\int e^{-4\sigma x_1}\hat B(4\sigma,x_1,0,\vc x''') : \vc e_2\otimes\vc\alpha\otimes\vc\alpha \\[5pt] \times (x_1-i\epsilon)^{-\frac{3}{2}}\dd x_1=0.
\end{equation}

The following lemma also appears in \cite{carfeiz}. 
\begin{lem}\label{jacobi}
Let $I$ be a bounded interval in $\R$ that does not contain the origin, $\mu>0$, and $\epsilon_0>0$. Let $f \in L^2(I)$ be such that 
\begin{equation}
\int_I f(t) (t-i\epsilon)^{-\mu}\dd t=0,\quad\forall \epsilon\in(0,\epsilon_0).
\end{equation}
Then $f= 0$.
\end{lem}

\begin{proof}
We have  $\forall \epsilon\in(0,\epsilon_0)$
\begin{multline}
0=\int_I f(t)(t-i\epsilon)^{-\mu}\dd t\\[5pt]
=\sum_{k=0}^\infty(i\epsilon)^k\frac{\mu(\mu+1)\cdots(\mu+k-1)}{k!}\int_I f(t)t^{-\mu} t^{-k}\dd t.
\end{multline}
It follows that for all $k=0,1,\ldots$
\begin{equation}
\int_I f(t)t^{-\mu} t^{-k}\dd t=0,
\end{equation}
so $f=0$, since $\mathrm{span}\,\{t^{-k}\}_{k=0}^\infty$ is dense in $L^2(I)$.
\end{proof}

Applying Lemma \ref{jacobi}, and undoing the Fourier transform in the first variable, we have that
\begin{equation}
B(x_0,x_1,0,\vc x''')=0,\quad \forall (x_0,x_1,0,\vc x''')\in\mathcal{O}.
\end{equation}
Since we are free to translate the origin of the coordinate system in the $x_2$ direction without affecting any of the above arguments, it follows that
\begin{equation}\label{three-two}
B(\vc x):\vc e_2\otimes\vc\alpha\otimes\vc\alpha=0,\quad\forall \vc x\in\mathcal{O}.
\end{equation}
Taking real and imaginary parts we get
\begin{equation}\label{re1}
B(\vc x):\vc e_2\otimes\vc e_0\otimes\vc e_0=B(\vc x):\vc e_2\otimes\vc e_1\otimes\vc e_1, \quad\forall \vc x\in\mathcal{O},
\end{equation}
\begin{equation}\label{im1}
B(\vc x):\vc e_2\otimes\vc e_0\otimes\vc e_1=-B(\vc x):\vc e_2\otimes\vc e_1\otimes\vc e_0, \quad\forall \vc x\in\mathcal{O}.
\end{equation}
Note that \eqref{im1} implies that $B$ is totally antisymmetric in all \emph{different} indices.

\subsection{The $\lambda^{-\frac{5}{2}}q_1^j$ term in \eqref{hormander}}

Let 
\begin{equation}
A_{jkl}=\frac{1}{3}\left(B_{jkl}+B_{klj}+B_{ljk}  \right).
\end{equation}
$A$ is fully antisymmetric, as for example
\begin{equation}
A_{jlk}=\frac{1}{3}\left(B_{jlk}+B_{lkj}+B_{kjl}  \right)
=\frac{1}{3}\left(-B_{ljk}-B_{klj}-B_{jkl}  \right)=-A_{jkl}, \; etc.
\end{equation}
As above we have that
\begin{multline}
\int A(\vc x):\nabla u_1(\vc x)\otimes\nabla u_2(\vc x)\otimes\nabla u_3(\vc x)\dd\vc x\\[5pt]
=\frac{1}{3}\int B(\vc x):\nabla u_1(\vc x)\otimes\nabla u_2(\vc x)\otimes\nabla u_3(\vc x)\dd\vc x\\[5pt]
+\frac{1}{3}\int B(\vc x):\nabla u_2(\vc x)\otimes\nabla u_3(\vc x)\otimes\nabla u_1(\vc x)\dd\vc x\\[5pt]
+\frac{1}{3}\int B(\vc x):\nabla u_3(\vc x)\otimes\nabla u_1(\vc x)\otimes\nabla u_2(\vc x)\dd\vc x
=0,
\end{multline}
for all  functions $u_1$,$u_2$, $u_3$, which are harmonic in $\mathcal{O}$.
The tensor $A$ is the antisymmetric part of $B$. In this subsection we will show that it is in fact zero.

Before proceeding with the computations, we remark that we can choose $h_1=1$. This will not affect our ability to remove the integration in $\vc x'''$ since the span of the products $h_2h_3$ is still dense. The effect of this choice will be that the $\delta^{-1}\dot\chi h_1\vc e_2+\chi\nabla'''h_1$ term will have no contribution once we impose the condition $x_2=0$.

The $\lambda^{-2}x_2^0$, $\lambda^{-1} x_2^2$, and $\lambda^0 x_2^4$ terms in $A :  V_{1;\tau}^+\otimes V_{2;\tau}^+\otimes\overline{V_{3;2\tau}^-}$ contribute to the $\lambda^{-\frac{5}{2}}$ order in \eqref{hormander}. Here we will only collect the terms that contain a $q_1^1$ factor.

The relevant coefficient of $\lambda^{-2}x_2^0$ in $A :  V_{1;\tau}^+\otimes V_{2;\tau}^+\otimes\overline{V_{3;2\tau}^-}$ is
\begin{multline}
\chi^3h_2h_3 A:\Big(
\dot v_{0;0}v_{0;1}^1\overline{v_{0;0}}\vc e_2\otimes\vc e_1\otimes\vc\alpha
-\frac{1}{2}\overline{\dot v_{0;0}}{v_{0;1}^1} v_{0;0}\vc e_2\otimes\vc\alpha\otimes\vc e_1
\Big)\\[5pt]
=\frac{q_1^1}{2}(x_1^2+\epsilon^2)^{-\frac{1}{2}}(x_1-i\epsilon)^{-\frac{3}{2}}\Big((x_1-i\epsilon)^{-1}+\frac{1}{2}(x_1+i\epsilon)^{-1}  \Big)\\[5pt]
\times\chi^3h_2h_3 A:\vc e_1\otimes\vc e_2\otimes\vc\alpha.
\end{multline}

The contributions to the coefficient of $\lambda^{-1}x_2^2$ can be obtained in the following ways: ``$\lambda^{-1}x_2^2\times 1\times 1$'', ``$\lambda^{-1}\times x_2^2\times 1$'', ``$\lambda^{-1}\times x_2\times x_2$'', and ``$\lambda^{-1}x_2\times x_2\times 1$''. The first of these, ``$\lambda^{-1}x_2^2\times 1\times 1$'', contributes no relevant terms.

The relevant terms contributed by ``$\lambda^{-1}\times x_2^2\times 1$'' are
\begin{multline}
\chi^3h_2h_3 A:\Big(
i|v_{0;0}|^2v_{0;1}^1\dot \psi_2\vc e_2\otimes\vc e_1\otimes\vc\alpha
+i|v_{0;0}|^2v_{0;1}^1\overline{\dot \psi}_2\vc e_2\otimes\vc\alpha\otimes\vc e_1\\[5pt]
+2i\dot v_{0;0}v_{0;1}^1\psi_2\overline{v_{0;0}}\vc e_2\otimes\vc e_1\otimes\vc\alpha
-i\overline{\dot v_{0;0}}v_{0;1}^1\psi_2{v_{0;0}}\vc e_2\otimes\vc \alpha\otimes\vc e_1
\Big)\\[5pt]
=\chi^3h_2h_3 A:\Big(
2i|v_{0;0}|^2v_{0;1}^1\dot \psi_2\vc e_2\otimes\vc e_1\otimes\vc\alpha
+i|v_{0;0}|^2v_{0;1}^1\overline{\dot \psi}_2\vc e_2\otimes\vc\alpha\otimes\vc e_1\\[5pt]
-i\overline{\dot v_{0;0}}v_{0;1}^1\psi_2{v_{0;0}}\vc e_2\otimes\vc \alpha\otimes\vc e_1
\Big).
\end{multline}
The relevant terms contributed by ``$\lambda^{-1}\times x_2\times x_2$'' are
\begin{equation}
\chi^3h_2h_3 A:\Big(
2i\dot v_{0;0}v_{0;1}^1\overline{v_{0;0}}\overline{\psi}_2\vc\alpha\otimes\vc e_1\otimes\vc e_2
-i\overline{\dot v_{0;0}}v_{0;1}^1 v_{0;0}\psi_2 \vc\alpha\otimes\vc e_2\otimes\vc e_1
\Big).
\end{equation}
The relevant terms contributed by ``$\lambda^{-1}x_2\times x_2\times 1$'' are
\begin{multline}
\chi^3h_2h_3 A:\Big(
2i|v_{0;0}|^2\psi_2\dot v_{0;1}^1\vc e_1\otimes\vc e_2\otimes\vc \alpha
+2i|v_{0;0}|^2\overline{\psi}_2\dot v_{0;1}^1\vc e_1\otimes\vc\alpha\otimes\vc e_2
\Big)\\[5pt]
=\chi^3h_2h_3 
2i|v_{0;0}|^2(\psi_2-\overline{\psi}_2)\dot v_{0;1}^1A:\vc e_1\otimes\vc e_2\otimes\vc \alpha.
\end{multline}
From the above we see that  the relevant coefficient of $\lambda^{-1}x_2^2$ in  $A :  V_{1;\tau}^+\otimes V_{2;\tau}^+\otimes\overline{V_{3;2\tau}^-}$ is
\begin{equation}
iq_1^1\epsilon^2(x_1^2+\epsilon^2)^{-\frac{5}{2}}(x_1-i\epsilon)^{-\frac{3}{2}}\chi^3h_2h_3 A:\vc e_1\otimes\vc e_2\otimes\vc\alpha.
\end{equation}

The contributions to the coefficient of $\lambda^{0}x_2^4$ can be obtained in the following ways: ``$x_2^4\times 1\times 1$'', ``$x_2^2\times x_2^2\times 1$'', ``$x_2^2\times x_2\times x_2$'', and ``$x_2^3\times x_2\times 1$''. The first of these, ``$x_2^4\times 1\times 1$'', contributes no relevant terms.

The relevant terms contributed by ``$x_2^2\times x_2^2\times 1$'' are
\begin{multline}
(-2)\chi^3h_2h_3 A:\Big(
|v_{0;0}|^2\psi_2\dot\psi_2 v_{0;1}^1\vc e_2\otimes\vc e_1\otimes\vc\alpha
+|v_{0;0}|^2\psi_2\overline{\dot\psi_2} v_{0;1}^1\vc e_2\otimes\vc\alpha\otimes\vc e_1
\Big)\\[5pt]
=(-2)\chi^3h_2h_3 
|v_{0;0}|^2\psi_2(\overline{\dot\psi}_2-\dot\psi_2) v_{0;1}^1A:\vc e_1\otimes\vc e_2\otimes\vc\alpha.
\end{multline}
The relevant terms contributed by ``$x_2^2\times x_2\times x_2$'' are
\begin{multline}
(-2)\chi^3h_2h_3 A:\Big(
|v_{0;0}|^2\dot\psi_2\overline{\psi}_2v_{0;1}^1\vc\alpha \otimes\vc e_1\otimes\vc e_2
+|v_{0;0}|^2\overline{\dot\psi}_2\psi_2 v_{0;1}^1 \vc\alpha\otimes\vc e_2\otimes\vc e_1
\Big)\\[5pt]
=(-2)\chi^3h_2h_3
|v_{0;0}|^2(\dot\psi_2\overline{\psi}_2-\overline{\dot\psi}_2\psi_2)v_{0;1}^1 A:\vc e_1 \otimes\vc e_2\otimes\vc\alpha.
\end{multline}
The relevant terms contributed by ``$x_2^3\times x_2\times 1$'' are
\begin{multline}
(-2)\chi^3h_2h_3 A:\Big(
|v_{0;0}|^2\dot\psi_2\psi_2 v_{0;1}^1\vc e_1\otimes\vc e_2\otimes\vc\alpha
+|v_{0;0}|^2\dot\psi_2\overline{\psi}_2v_{0;1}^1\vc e_1\otimes\vc\alpha\otimes\vc e_2
\Big)\\[5pt]
=(-2)\chi^3h_2h_3 
|v_{0;0}|^2\dot\psi_2(\psi_2-\overline{\psi}_2) v_{0;1}^1A:\vc e_1\otimes\vc e_2\otimes\vc\alpha.
\end{multline}
From the above we see that  the coefficient of $\lambda^{0}x_2^4$ in  $A :  V_{1;\tau}^+\otimes V_{2;\tau}^+\otimes\overline{V_{3;2\tau}^-}$  is \emph{zero}.

Note  that there are several terms of order $\lambda^{-2}$ in $L_0(U)+\lambda^{-1}L_1(U)+\lambda^{-2}L_2(U)$ that involve derivatives of $A$. However, none contain a factor of $q_1^j$.

We can now determine the coefficient of $\lambda^{-2}q_1^1$ in $L_0(U)+\lambda^{-1}L_1(U)+\lambda^{-2}L_2(U)$.  After some grouping of terms we find it to be
\begin{multline}
\frac{1}{8}(x_2^2+\epsilon^2)^{-\frac{1}{2}}(x_1-i\epsilon)^{-\frac{3}{2}}\left[5(x_1-i\epsilon)^{-1}+(x_1+i\epsilon)^{-1}  \right]\\[5pt]\times A|_{x_2=0}:\vc e_1\otimes\vc e_2\otimes\vc\alpha.
\end{multline}
It follows then  that
\begin{multline}
\int e^{-4\sigma x_1}\hat A(4\sigma,x_1,0,\vc x''') : \vc e_1\otimes\vc e_2\otimes\vc\alpha\\[5pt]
\times(x_1-i\epsilon)^{-\frac{3}{2}}\left[5(x_1-i\epsilon)^{-1}+(x_1+i\epsilon)^{-1}  \right]\\[5pt]
\times h_1(\vc x''') h_2(\vc x''') h_3(\vc x''')\dd x_1\dd\vc x'''=0
\end{multline}
The only thing needed in order to conclude that $A=0$ is an analogue to Lemma \ref{jacobi}, so that we may handle the integration in $x_1$. 
\begin{lem}\label{jacobi2}
Let $I$ be a bounded interval in $\R$ that does not contain the origin,  and let $\epsilon_0>0$. Let $f \in L^2(I)$ be such that 
\begin{equation}
\int_I f(t) (t-i\epsilon)^{-\frac{3}{2}}
\left[5(t-i\epsilon)^{-1}+(t+i\epsilon)^{-1}  \right]\dd t=0,\quad\forall \epsilon\in(0,\epsilon_0).
\end{equation}
Then $f= 0$.
\end{lem}
\begin{proof}
As before, we only need to show that the Taylor expansion at $z=0$ of the function $(1-z)^{-\frac{3}{2}}\left[5(1-z)^{-1}+(1+z)^{-1}  \right]$ does not contain any zero coefficients. In fact, the coefficient of order $j$ is
\begin{equation}
\sum_{l=0}^j \left[5+(-1)^{j-l}  \right]\frac{(2l+1)!!}{2^ll!}>0.
\end{equation}
\end{proof}

In the same way that we have arrived at \eqref{three-two}, we can now  conclude that $A=0$.
It follows that $B_{jkl}=0$ when $j\neq k\neq l\neq j$, since these coefficients are already fully antisymmetric, as we have already pointed out following equation \eqref{im1}.

The only possibly non-zero coefficients remaining are $b_j=B_{jkk}=-B_{kjk}\in C_0^\infty(\mathcal{O})$, whith $j\neq k$. Then $B$ has the following structure
\begin{equation}
B_{jkl}=b_j\delta_{kl}(1-\delta_{jk})-b_k\delta_{jl}(1-\delta_{jk})
=b_j\delta_{kl}-b_k\delta_{jl}.
\end{equation}
This structure is sufficiently simple that we can use the harmonic functions constructed by Calder\'on in \eqref{basic-integral-identity}.
Let 
\begin{equation}
u_1=u_3=e^{i\vc\zeta_{+}\cdot \vc x},\quad u_2=e^{2i\vc\zeta_{-}\cdot \vc x}.
\end{equation}
Then \eqref{basic-integral-identity} gives that
\begin{equation}\label{calderon1}
|\vc\xi|^2\hat{\vc b}(\vc\xi)\cdot\vc\zeta_+=0.
\end{equation}
Here $\hat{\vc b}$ denotes the Fourier transform of $\vc b$ in all variables.
Similarly, if we make the choices
\begin{equation}
u_1=u_3=e^{i\vc\zeta_{-}\cdot \vc x},\quad u_2=e^{2i\vc\zeta_{+}\cdot \vc x},
\end{equation}
we obtain
\begin{equation}\label{calderon2}
|\vc\xi|^2\hat{\vc b}(\vc\xi)\cdot\vc\zeta_-=0.
\end{equation}
Adding \eqref{calderon1} and \eqref{calderon2} we get
\begin{equation}
|\vc\xi|^2\hat{\vc b}(\vc\xi)\cdot\vc\xi=0.
\end{equation}
Subtracting the two we get
\begin{equation}
|\vc\xi|^3\hat{\vc b}(\vc\xi)\cdot\vc\nu=0.
\end{equation}
Since $\vc\nu$ can be any unit vector orthogonal to $\vc\xi$, it follows that $\hat{\vc b}=0$, so $\vc b=0$. Therefore $B=0$.

\section{An application to inverse problems for quasilinear systems}\label{application}

In order to show the applicability of our result, we will here consider an application to an inverse boundary value problem for an elliptic quasilinear system, with ``anisotropic'' nonlinearity. This is not necessarily meant to model a particular kind of physical behavior, but rather to suggest the kind of applications that our density result may have.

To be precise, consider the following system of two coupled equations in a bounded, $C^{1,1}$ domain $\dom\subset\R^{1+n}$
\begin{equation}\label{sys}
\left\{\begin{array}{l} -\Delta u^J+\sum_{K=1}^2\nabla(A^{JK}:\nabla u^J\otimes\nabla u^K)=0,\quad J=1,2,\\[5pt] (u^1,u^2)|_{\pr\dom}=(f^1,f^2).\end{array}\right.
\end{equation}
Each $A^{JK}=(A^{JK}_{jkl})_{j,k,l=0}^n$ is a 3-tensor which we assume to be $C^\infty(\dom)$. The $A^{JJ}$ tensors would be symmetric in the last two indices, but the $A^{JK}$, $J\neq K$, do not a priori need to have any symmetry at all.

First we need to discuss the forward problem. What follows is a standard contraction principle argument. We sketch it out for the sake of completeness. Let $p\in(1+n,\infty)$. For $f^J\in W^{2-\frac{1}{p},p}(\pr\dom)$, let $v^J_0\in W^{2,p}(\dom)$ be the solution to the problem
\begin{equation}
\left\{\begin{array}{l}\Delta v^J_0=0,\quad J=1,2,\\[5pt] v^J_0|_{\pr\dom}=f^J.\end{array}\right.
\end{equation}
We have that
\begin{equation}
||v^J_0||_{W^{2,p}(\dom)}\leq C||f^J||_{W^{2-\frac{1}{p},p}(\pr\dom)},\quad J=1,2.
\end{equation}
Let $G_0: L^p(\dom)\to W^{2,p}(\dom)\cap W^{1,p}_0(\dom)$ be the solution operator to the equation $\Delta u=F$, with zero Dirichlet boundary conditions. This is a bounded linear operator. We define the map
\begin{equation}
[\mathcal{T}(v^1,v^2)]^J=v^J_0-G_0\left(\sum_{K=1}^2\nabla(A^{JK}:\nabla v^J\otimes\nabla v^K)  \right),\quad J=1,2.
\end{equation}
It is easy to check, using Sobolev embedding,  that $\mathcal{T}:W^{2,p}(\dom)\times W^{2,p}(\dom)\to W^{2,p}(\dom)\times W^{2,p}(\dom)$ and
\begin{equation}\label{T-estimate}
||[\mathcal{T}(v^1,v^2)]^J||_{W^{2,p}(\dom)}
\leq C\left(||f^J||_{W^{2-\frac{1}{p},p}(\pr\dom)} +||v^1||_{W^{2,p}(\dom)}^2 +  ||v^2||_{W^{2,p}(\dom)}^2\right).
\end{equation}
Assuming $||f||_{W^{2-\frac{1}{p},p}(\pr\dom)}<1/8(1+C)$, we see that $\mathcal{T}$ maps the ball of radius $1/4(1+C)$ in $W^{2,p}(\dom)\times W^{2,p}(\dom)$ to itself.

If $(v^1,v^2),(w^1,w^2)\in W^{2,p}(\dom)\times W^{2,p}(\dom)$, we have that
\begin{multline}
||[\mathcal{T}(v^1,v^2)-\mathcal{T}(w^1,w^2)]^J||_{W^{2,p}(\dom)}\\[5pt]
\leq C\left(||v^1||_{W^{2,p}(\dom)}+||v^2||_{W^{2,p}(\dom)}+||w^1||_{W^{2,p}(\dom)}+||w^2||_{W^{2,p}(\dom)}\right)\\[5pt]
\times\left(||v^1-w^1||_{W^{2,p}(\dom)}+||v^2-w^2||_{W^{2,p}(\dom)}  \right),
\end{multline}
so $\mathcal{T}$ is a contraction on the ball of radius $1/4(1+C)$ in $W^{2,p}(\dom)\times W^{2,p}(\dom)$. The unique fixed point of $\mathcal{T}$ is the solution to \eqref{sys}. This solution $(u^1,u^2)$ then satisfies the estimate
\begin{equation}\label{sys-estimate}
||u^J||_{W^{2,p}(\dom)}\leq C\left(||f^1||_{W^{2-\frac{1}{p},p}(\pr\dom)}+||f^2||_{W^{2-\frac{1}{p},p}(\pr\dom)}\right),\quad J=1,2.
\end{equation}

Once we know \eqref{sys} has unique solutions for at least sufficiently small boundary data, we can define the Dirichlet-to-Neumann map $\Lambda:W^{2-\frac{1}{p},p}(\pr\dom)\times W^{2-\frac{1}{p},p}(\pr\dom)\to W^{1-\frac{1}{p},p}(\pr\dom)\times W^{1-\frac{1}{p},p}(\pr\dom)$  to be
\begin{equation}
[\Lambda(f^1,f^2)]^J=\vc n\cdot\nabla u^J|_{\pr\dom},\quad J=1,2,
\end{equation}
where $\vc n$ is the outer pointing unit normal vector on $\pr\dom$.

We can prove the following uniqueness result.

\begin{thm}\label{inverse-thm}
Suppose $A^{JK}$, $\tilde{A}^{JK}$, are two sets of coefficients as above, and that $\Lambda=\Lambda'$. Then $A^{JK}=\tilde{A}^{JK}$, $J,K=1,2$.
\end{thm}
\begin{proof}
The usual method used in proving this kind of result is the so called ``second linearization'' trick, first used in \cite{I1}. In order to apply it here, consider $\epsilon>0$, sufficiently small. Let $u_\epsilon^J$, $J=1,2$, be the solution in $\dom$ to 
\begin{equation}
\left\{\begin{array}{l} -\Delta u^J_\epsilon+\sum_{K=1}^2\nabla(A^{JK}:\nabla u^J_\epsilon\otimes\nabla u^K_\epsilon)=0,\quad J=1,2,\\[5pt] (u^1_\epsilon,u^2_\epsilon)|_{\pr\dom}=\epsilon(f^1,f^2).\end{array}\right.,
\end{equation}
for some pair $f^1,f^2\in W^{2-\frac{1}{p},p}(\pr\dom)$ of Dirichlet data. By \eqref{T-estimate} and \eqref{sys-estimate} we have that
\begin{equation}
u^J_\epsilon=\epsilon v^J_0+\mathscr{O}(\epsilon^2),
\end{equation}
where the $\mathscr{O}(\epsilon^2)$ should be understood in the sense of $W^{2,p}(\dom)$ norms, and the $v^J_0$ are as defined above. Let $v_1^J$ satisfy
\begin{equation}
\left\{\begin{array}{l} -\Delta v^J_1+\sum_{K=1}^2\nabla(A^{JK}:\nabla v^J_0\otimes\nabla v^K_0)=0,\quad J=1,2,\\[5pt] (v^1_1,v^2_1)|_{\pr\dom}=0.\end{array}\right..
\end{equation}
We can also state this as
\begin{equation}
(v^1_1,v^2_1)=\mathcal{T}(v^1_0,v^2_0)-(v^1_0,v^2_0)=-G_0\left(\sum_{K=1}^2\nabla(A^{JK}:\nabla v^J_0\otimes\nabla v^K_0)\right).
\end{equation}
Then
\begin{equation}
(u^1_\epsilon,u^2_\epsilon)=\mathcal{T}^2(u^1_\epsilon,u^2_\epsilon)
=\epsilon(v^1_0,v^2_0)+\epsilon^2(v^1_1,v^2_1)
+\mathscr{O}(\epsilon^3).
\end{equation}

It follows that the Dirichlet-to-Neumann map has the following expansion in $\epsilon$
\begin{equation}
[\Lambda(f^1,f^2)]^J=\epsilon\vc n\cdot\nabla v_0^J+\epsilon^2\vc n\cdot\nabla v^J_1+\mathscr{O}(\epsilon^3).
\end{equation}
Suppose $w$ is a smooth harmonic function in a negihborhood of $\dom$. Then
\begin{multline}
\la [\Lambda(f^1,f^2)]^J,w|_{\pr\dom}\ra
=\epsilon \int \nabla v_0^J(\vc x)\cdot\nabla w(\vc x)\dd\vc x\\[5pt]+\epsilon^2\sum_{K=1}^2\int A^{JK}(\vc x):\nabla w(\vc x)\otimes\nabla v^J_0(\vc x)\otimes\nabla v^K_0(\vc x)\dd\vc x+\mathscr{O}(\epsilon^3).
\end{multline}
By the assumtions of Theorem \ref{inverse-thm}, since each order in $\epsilon$ must match in the two Dirichlet-to-Neumann maps, we have that
\begin{equation}\label{ajk}
\sum_{K=1}^2\int (A^{JK}-\tilde{A}^{JK})(\vc x):\nabla w(\vc x)\otimes\nabla v^J_0(\vc x)\otimes\nabla v^K_0(\vc x)\dd\vc x=0.
\end{equation}

Choosing $f^2=0$, we are left with the integral identity
\begin{equation}
\int (A^{11}-\tilde{A}^{11})(\vc x):\nabla w(\vc x)\otimes\nabla v^1_0(\vc x)\otimes\nabla v^1_0(\vc x)\dd\vc x=0.
\end{equation}
This case has been treated in \cite{carfeiz}. Since $A^{11}$ and $\tilde{A}^{11}$ are symmetric in the last two indices, by polarization we can conclude that
\begin{equation}
\int (A^{11}-\tilde{A}^{11})(\vc x):\nabla u_1(\vc x)\otimes\nabla u_2(\vc x)\otimes\nabla u_3(\vc x)\dd\vc x=0,
\end{equation}
for any three harmonic functions $u_1$, $u_2$, and $u_3$. By Theorem \ref{integral-thm} it follows that
$
A^{11}=\tilde{A}^{11}
$. Similarly we have that $A^{22}=\tilde{A}^{22}$.

Returning to \eqref{ajk}, we have that
\begin{equation}
\int (A^{12}-\tilde{A}^{12})(\vc x):\nabla w(\vc x)\otimes\nabla v^1_0(\vc x)\otimes\nabla v^2_0(\vc x)\dd\vc x=0.
\end{equation}
Since the harmonic functions $w$, $v^1_0$, and $v^2_0$ can be chosen arbitrarily and independently we can conclude that $A^{12}=\tilde{A}^{12}$, ans similarly that $A^{21}=\tilde{A}^{21}$. This concludes the proof of the theorem.
\end{proof}


\section{The case of products of two gradients of harmonic functions}\label{calderon}

Here we will investigate the analogous problem for tensor products of gradients of two harmonic functions. This question is related to the linearized anisotropic Calder\'on problem, i.e. the equivalent of the result of \cite{Ca} in the case of anisotropic conductivities. Because of this connection, we will first take the time to describe this inverse problem before giving a proof to Theorem \ref{lin-cal}.

In the domain $\dom$ consider equations of the type
\begin{equation}\label{eq-cond}
\left\{\begin{array}{l}\sum_{j,k=0}^n\pr_j(A_{jk}\pr_k u)=0,\\[5pt] u|_{\pr\dom}=f.\end{array}\right.
\end{equation}
We assume that the coefficients $A_{jk}\in C^\infty(\R^n)$ are elliptic, i.e. there exists $\lambda>0$ such that
\begin{equation}
 A:\vc\xi\otimes\vc\xi\geq  \lambda |\vc\xi|^2,\quad\forall\vc\xi\in\R^{1+n}.
\end{equation}
Here the Dirichlet-to-Neumann map $\Lambda_A:H^{\frac{1}{2}}(\pr\dom)\to H^{-\frac{1}{2}}(\pr\dom)$ is
\begin{equation}
\Lambda_A(f)=A:\vc n\otimes\nabla u|_{\pr\dom},
\end{equation}
where $u$ is the solution of \eqref{eq-cond}. 

As in section \ref{application}, we are interested in the question of uniqueness for the inverse boundary value problem: if $A$ and $\tilde{A}$ are both as above and furthermore $\Lambda_A=\Lambda_{\tilde{A}}$, does it follow that $A=\tilde{A}$? This  turns out not to be the case. It was first noted by Luc Tartar (account given in \cite{KV}) that if $\Phi:\overline{\dom}\to\overline{\dom}$ is a smooth diffeomorphism such that $\Phi(\vc x)=\vc x$ for any $\vc x\in\pr\dom$, and if
\begin{equation}
 \tilde{A}(\vc x)=\left[\det D\Phi(\vc x)\right]^{-1}(D\Phi(\vc x))^t A(\Phi^{-1}(\vc x)) (D\Phi(\vc x)),
\end{equation}
then $\Lambda_{\tilde{A}}=\Lambda_A$. The question then becomes: is this the only obstruction to uniqueness? In dimension $1+n\geq 3$, there is so far no known answer.

The dependence of $\Lambda_A$ on $A$ is nonlinear and difficult to characterize in a useful way. In \cite{Ca}, Calder\'on has considered the linearized version of the problem for isotropic coefficients (i.e. $A_{jk}(\vc x)=\gamma(\vc x)\delta_{jk}$). Here we will sketch out how this idea works when the coefficients  are anisotropic. For $A$ as above, let $\Lambda'_I(A)$ be the Frech\'et derivative of the Dirichlet-to-Neumann map in the direction of $A$
\begin{equation}
\Lambda'_I(A)=w-\lim_{t\to0}\frac{1}{t}\left(\Lambda_{I+tA}-\Lambda_I\right),
\end{equation}
where $I$ is the matrix with coefficients $I_{jk}=\delta_{jk}$. In order to identify $\Lambda'_I(A)$, let $f\in H^{\frac{1}{2}}(\pr\dom)$ and $u_t$ be the solution to the problem
\begin{equation}
\left\{\begin{array}{l} \sum_{j,k=0}^n\pr_j\left[(\delta_{jk}+tA_{jk})\pr_k u_t\right]=0,\\[5pt] u_t|_{\pr\dom}=f.\end{array}\right.
\end{equation}
By well known elliptic estimates
\begin{equation}
||u_t||_{H^1(\dom)}\leq C||f||_{H^{\frac{1}{2}}(\dom)}.
\end{equation}
If we make the Ansatz
\begin{equation}
u_t=u_0+tw,
\end{equation}
then $w$ must satisfy
\begin{equation}
\left\{\begin{array}{l} \Delta w+\sum_{j,k=0}^n\pr_j\left(A_{jk}\pr_k u_t\right)=0,\\[5pt] w|_{\pr\dom}=0.\end{array}\right.
\end{equation}
It follows that
\begin{equation}
||w||_{H^1(\dom)}\leq C||u_t||_{H^1(\dom)}\leq C||f||_{H^{\frac{1}{2}}(\dom)}.
\end{equation}
Suppose that $g\in H^{\frac{1}{2}}(\pr\dom)$ and $v$ is the harmonic function whose trace on $\pr\dom$ is $g$. It is clear that
\begin{equation}
\la \Lambda'_I(A)f,g\ra=\int_{\pr\dom}A(\vc x):\nabla u_0(\vc x)\otimes\nabla v(\vc x)\dd\vc x.
\end{equation}
Therefore, if $\Lambda'_I(A)=\Lambda'_I(\tilde{A})$ we have that
\begin{equation}
\int_{\pr\dom}\left(A(\vc x)-\tilde{A}(\vc x)\right):\nabla u_1(\vc x)\otimes\nabla u_2(\vc x)\dd\vc x=0,
\end{equation}
for all harmonic functions $u_1$, $u_2$.

In order to identify the expected obstruction to uniqueness in the linearized anisotropic Calder\'on problem, suppose $\vc v\in C^\infty(\dom;\R^n)$ is such that $\vc v|_{\pr\dom}=0$ and let $\Phi_t$ be the flow generated by $\vc v$. Then
\begin{equation}
\Phi_t(\vc x)=\vc x+t\vc v(\vc x)+o(t),
\end{equation}
\begin{equation}
[D\Phi_t(\vc x)]_{jk}=\delta_{jk}+t\pr_j v_k(\vc x)+o(t),
\end{equation}
\begin{equation}
\det D\Phi_t(\vc x)=1+t\nabla\cdot\vc v+o(t).
\end{equation}
For a matrix $A(\vc x)$  let $A_t(\vc x)$ be the family of  matrices
\begin{equation}
A_t(\vc x)=\left[\det D\Phi_t(\vc x)\right]^{-1}(D\Phi_t(\vc x))^t A(\Phi_t^{-1}(\vc x))( (D\Phi_t(\vc x)).
\end{equation}
For the identity matrix $I$ and small $t$ we have
\begin{multline}
(I_t)_{jk}=(1-t\nabla\cdot\vc v)(\delta_{lj}+t\pr_l v_j)\delta_{lm}(\delta_{mk}+t\pr_m v_k)+o(t)\\[5pt]
=(1-t\nabla\cdot\vc v)\left[\delta_{jk}+t(\pr_j v_k+\pr_k v_j)\right]+o(t)\\[5pt]
=\delta_{jk}+t(\pr_j v_k+\pr_k v_j-\delta_{jk}\nabla\cdot\vc v)+o(t).
\end{multline}
Then
\begin{multline}
\Lambda'_I(A)=\left.w-\frac{\dd}{\dd t}\right|_{t=0}\Lambda_{I+tA}
=\left.\frac{\dd}{\dd t}\right|_{t=0}\Lambda_{(I_t+tA_t)}\\[5pt]
=\Lambda'_I\left(\left.\frac{\dd}{\dd t}\right|_{t=0}I_t+A+0\times\left.\frac{\dd}{\dd t}\right|_{t=0}A_t\right)\\[5pt]
=\Lambda'_I\left(\left.\frac{\dd}{\dd t}\right|_{t=0}I_t+A\right).
\end{multline}
So the expectation should be that uniqueness holds in the linearized anisotropic Calder\'on problem only up to a 
\begin{equation}
\pr_j v_k+\pr_k v_j-\delta_{jk}\nabla\cdot\vc v
\end{equation}
term, as claimed in Theorem \ref{lin-cal}.

\begin{proof}[Proof of Theorem \ref{lin-cal}]
We begin by proving sufficiency. 
Let $u$ be a harmonic function on $\R^{1+n}$ and $\vc v\in W^{1,p}_0(\dom)$. Then
\begin{multline}
\sum_{j,k=0}^n\int_\dom(\pr_j v_k(\vc x)+\pr_k v_j(\vc x))\pr_j u(\vc x)\pr_k u(\vc x)\dd\vc x\\[5pt]
=2\sum_{j,k=0}^n\int_\dom \pr_j v_k(\vc x)\pr_j u(\vc x)\pr_k u(\vc x)\dd\vc x\\[5pt]
=-2\sum_{j,k=0}^n\int_\dom v_k(\vc x)\pr_j u(\vc x)\pr_j\pr_k u(\vc x)\dd\vc x\\[5pt]
=-\sum_{j,k=0}^n\int_\dom v_k(\vc x)\pr_k\left(\pr_j u(\vc x)\pr_j u(\vc x)\right)\dd\vc x\\[5pt]
=\sum_{j=0}^n\int_\dom(\nabla\cdot\vc v(\vc x))\pr_j u(\vc x)\pr_j u(\vc x)\dd\vc x,
\end{multline}
so
\begin{equation}
\sum_{j,k=0}^n\int_\dom \left(\pr_j v_k(\vc x)+\pr_k v_j(\vc x)-\delta_{jk}\nabla\cdot\vc v(\vc x)\right)\pr_j u(\vc x)\pr_k u(\vc x)\dd\vc x=0,
\end{equation}
for all harmonic functions $u$. By polarization we have that
\begin{equation}
\sum_{j,k=0}^n\int_\dom \left(\pr_j v_k(\vc x)+\pr_k v_j(\vc x)-\delta_{jk}\nabla\cdot\vc v(\vc x)\right)\pr_j u_1(\vc x)\pr_k u_2(\vc x)\dd\vc x=0,
\end{equation}
for all harmonic functions $u_1$, $u_2$.

Now let $a_{jk}$ be as in the conclusion of the theorem. For $u_1$, $u_2$ harmonic functions on $\R^{1+n}$ we have
\begin{multline}
\sum_{j,k=0}^n\int_\dom a_{jk}(\vc x)\pr_j u_1(\vc x)\pr_k u_2(\vc x)\dd\vc x
=\sum_{j,k=0}^n\;\int_{\R^{1+n}} a_{jk}(\vc x)\pr_j \left[u_1(\vc x)\pr_k u_2(\vc x)\right]\dd\vc x\\[5pt]
-\sum_{j,k=0}^n\;\int_{\R^{1+n}} a_{jk}(\vc x) u_1(\vc x)\pr_j \pr_k u_2(\vc x) \dd\vc x
=-\sum_{j,k=0}^n\la \pr_j a_{jk},u_1\pr_k u_2\ra=0.
\end{multline}

The rest of the proof will show necessity. It helps to first split $C$ into symmetric and antisymmetric parts $C=C^s+C^a$, where
\begin{equation}
C^s_{jk}=\frac{1}{2}(C_{jk}+C_{kj}),\quad C^a_{jk}=\frac{1}{2}(C_{jk}-C_{kj}), \quad j,k=0,\ldots,n.
\end{equation}
As we have seen in the Introduction, it is easy to show that both $C^s$ and $C^a$ have the property \eqref{2-int-id}. If we choose Calder\'on type solutions
\begin{equation}
u_1=e^{i\vc \zeta_{+}\cdot \vc  x},\quad u_2=e^{i\vc \zeta_{-}\cdot \vc  x},
\end{equation}
in the identity for $C^a$, a simple computation gives that 
\begin{equation}
i|\vc \xi|\hat C^a(\vc \xi):\vc\eta\otimes\vc\xi=0,\quad\forall \vc\xi,\vc\eta,
\end{equation}
where we have identified $C^a$ with its extension by zero to the whole $\R^{1+n}$.
Since we also have that $\hat C^a(\vc\xi):\vc\xi\otimes\vc\xi=0$, it follows that
\begin{equation}
\sum_{j=0}^n\hat C^a(\vc\xi)_{jk}\xi_j=0.
\end{equation}
By inverting the Fourier transform we have  that $a_{jk}=C^a_{jk}$.

Let $\mathcal{O}$ be a neighborhood of $\overline{\dom}$. We will first assume that $C^s\in C_0^\infty(\mathcal{O})$ and satisfies 
\begin{equation}
\int_{\R^{1+n}} C^s(\vc x):\nabla u_1(\vc x)\otimes\nabla u_2(\vc x)\dd\vc x=0,
\end{equation}
for all $u_1$, $u_2$ harmonic functions in $\R^{1+n}$.
Let
\begin{equation}
v_j=\sum_{k=0}^nG_0(\pr_k C^s_{jk}),
\end{equation}
where $G_0$ is the Dirichlet Green's operator for the Laplacian on $\mathcal{O}$. That is, if 
\begin{equation}
\left\{\begin{array}{l}\Delta u=F\in H^{-1}(\dom),\\[5pt] u|_{\pr\dom}=0,\end{array}\right.
\end{equation}
then $G_0(F)=u\in H^1_0(\dom)$.
Define
\begin{multline}
B_{jk}=C^s_{jk}-\pr_j v_k-\pr_k v_j+\delta_{jk}\nabla\cdot\vc v\\[5pt]
=C^s_{jk}-\sum_{l=0}^nG_0(\pr_j\pr_lC^s_{kl})-\sum_{l=0}^nG_0(\pr_k\pr_lC^s_{jl})+\delta_{jk}\sum_{l,m=0}^n G_0(\pr_m\pr_lC^s_{lm}).
\end{multline}
It is easy to check that $B\in C_0^\infty(\mathcal{O};\R^{n\times n})$ and
\begin{multline}\label{sol}
\sum_{k=0}^n\pr_k B_{jk}
=\sum_{k=0}^n\pr_k C^s_{jk} -\sum_{l,m=0}^nG_0(\pr_j\pr_m\pr_l C^s_{ml})\\[5pt] -\sum_{k=0}^nG_0(\Delta\pr_k C^s_{jk}) +\sum_{l,m=0}^nG_0(\pr_j\pr_m\pr_l C^s_{ml})
=0,
\end{multline}
where the last step follows from 
\begin{equation}
\sum_{k=0}^n G_0(\Delta\pr_k C^s_{jk})=\pr_k C^s_{jk},
\end{equation}
which holds by uniqueness of solutions for the Poisson equation and the fact that $\sum_{k=0}^n\pr_k C^s_{jk}\in C_0^\infty(\mathcal{O})$.

It is clear that
\begin{equation}
\int_{\R^{1+n}} B(\vc x):\nabla u_1(\vc x)\otimes\nabla u_2(\vc x)\dd\vc x=0,
\end{equation}
for all $u_1$, $u_2$ harmonic functions in $\R^{1+n}$.
If we choose  again 
\begin{equation}
u_1=e^{i\vc \zeta_{+}\cdot \vc  x},\quad u_2=e^{i\vc \zeta_{-}\cdot \vc  x},
\end{equation}
 we obtain  that (here $B$ is identified with its extension by zero to $\R^{1+n}$)
\begin{equation}
\hat B(\vc \xi):\vc \zeta_+\otimes\vc \zeta_-=0,\quad \forall\vc \xi,\vc \mu.
\end{equation}
This is equivalent to
\begin{equation}\label{equiv}
\hat B(\vc \xi):\left(\vc \xi\otimes\vc \xi+|\vc \xi|^2\vc \mu\otimes\vc \mu\right)=0,\quad \forall\vc \xi,\vc \mu.
\end{equation}
Using \eqref{sol} and polarization
\begin{equation}
\hat B(\vc \xi):\vc\mu_1\otimes\vc\mu_2=0,\quad\forall\vc\mu_1,\vc\mu_2\perp\vc\xi,
\end{equation}
\begin{equation}
\hat B(\vc\xi):\vc\xi\otimes\vc\xi=\hat B(\vc\xi):\vc\xi\otimes\vc\mu=0, \forall\vc \xi\perp\vc \mu.
\end{equation}
So $B=0$.

It remains to remove the regularity restriction on $C^s$. Suppose now that $C^s\in L^p(\dom)$.
For $\vc y\in\R^{1+n}$ small enough we have that
\begin{equation}
\int C^s(\vc x-\vc y):\nabla u_1(\vc x)\otimes\nabla u_2(\vc x)\dd \vc x=0,
\end{equation}
for all smooth harmonic functions $u_1$, $u_2$  in $\R^{1+n}$.  As in the introduction, this is the case because a translation  of a harmonic function is still harmonic. Suppose now that $\varphi_\epsilon$ is a compactly supported smooth approximation of identity. It follows that
\begin{multline}
\int \left(\varphi_\epsilon*C^s\right)(\vc x):\nabla u_1(\vc x)\otimes\nabla u_2(\vc x)\dd \vc x\\[5pt]=
\int \varphi_\epsilon(\vc y)\left(\int C^s(\vc x-\vc y):\nabla u_1(\vc x)\otimes\nabla u_2(\vc x)\dd \vc x\right)\dd \vc y\\[5pt]=0,
\end{multline}
for all smooth harmonic functions $u_1$, $u_2$  in $\R^{1+n}$. As we have seen above, we have that
\begin{equation}
\varphi_\epsilon* C^s_{jk}=\pr_jv^\epsilon_k+\pr_k v^\epsilon_j-\delta_{jk}\nabla\cdot\vc v^\epsilon,
\end{equation}
where
\begin{equation}
v_j^\epsilon=\sum_{k=0}^nG_0(\pr_k \varphi_\epsilon*C^s_{jk}).
\end{equation}
Since $\varphi_\epsilon*C^s$ is Cauchy in $L^p(\mathcal{O})$, it follows that $\vc v^\epsilon$ is Cauchy in $W^{1,p}_0(\mathcal{O})$ and therefore has a limit $\vc v$. Taking limits we have that
\begin{equation}
C^s_{jk}=\pr_jv_k+\pr_k v_j-\delta_{jk}\nabla\cdot\vc v.
\end{equation}
Since $\mathcal{O}$ can be chosen arbitrarily, it follows that $\vc v|_{\R^{1+n}\setminus\overline{\dom}}=0$. Therefore $\vc v|_{\pr\dom}=0$. This ends the proof.
\end{proof}

\paragraph{Acknowledgments:} C\u at\u alin I. C\^arstea was supported by NSF of China under grant 11931011. Ali Feizmohammadi was supported by EPSRC grant EP/P01593X/1.

\bibliography{harmonic}

\begin{thebibliography}{10}

\bibitem{BU}
V.~M. Babich and V.V. Ulin.
\newblock Complex space-time ray method and “quasifotons”.
\newblock {\em Zapiski Nauchnykh Seminarov POMI}, 117:5--12, 1981.

\bibitem{buu}
A.~L. Bukhgeǐm and G.~Uhlmann.
\newblock Recovering a potential from partial {C}auchy data.
\newblock {\em Communications in Partial Differential Equations}, 27:653, 2002.

\bibitem{Ca}
A.~P. Calder{\'o}n.
\newblock On an inverse boundary value problem.
\newblock In {\em Seminar on {N}umerical {A}nalysis and its {A}pplications to
  {C}ontinuum {P}hysics ({R}io de {J}aneiro, 1980)}, pages 65--73. Soc. Brasil.
  Mat., Rio de Janeiro, 1980.

\bibitem{C}
C.~I. C{\^a}rstea.
\newblock On an inverse boundary value problem for a nonlinear time harmonic
  {M}axwell system.
\newblock {\em arXiv preprint arXiv:1804.09586}, 2018.

\bibitem{carfeiz}
C.~I. C{\^a}rstea and A.~Feizmohammadi.
\newblock An inverse boundary value problem for certain anisotropic quasilinear
  elliptic equations.
\newblock {\em arXiv preprint arXiv:2008.04517}, 2020.

\bibitem{CK}
C.~I. C{\^a}rstea and M.~Kar.
\newblock Recovery of coefficients for a weighted p-{L}aplacian perturbed by a
  linear second order term.
\newblock {\em arXiv preprint arXiv:2001.01436}, 2020.

\bibitem{CNV}
C.~I. C{\^a}rstea, G.~Nakamura, and M.~Vashisth.
\newblock Reconstruction for the coefficients of a quasilinear elliptic partial
  differential equation.
\newblock {\em Applied Mathematics Letters}, 2019.

\bibitem{DKLS}
D.~Dos Santos~Ferreira, Y.~Kurylev, M.~Lassas, and M.~Salo.
\newblock The {C}alder{\'o}n problem in transversally anisotropic geometries.
\newblock {\em Journal of the European Mathematical Society},
  18(11):2579--2626, 2016.

\bibitem{EPS}
H.~Egger, J.-F. Pietschmann, and M.~Schlottbom.
\newblock Simultaneous identification of diffusion and absorption coefficients
  in a quasilinear elliptic problem.
\newblock {\em Inverse Problems}, 30(3):035009, 2014.

\bibitem{FO}
A.~Feizmohammadi and L.~Oksanen.
\newblock An inverse problem for a semi-linear elliptic equation in
  {R}iemannian geometries.
\newblock {\em Journal of Differential Equations}, 269(6):4683--4719, 2020.

\bibitem{fksu}
D.~D.~S. Ferreira, C.~E. Kenig, J.~Sj{\"o}strand, and G.~Uhlmann.
\newblock On the linearized local calder{\'o}n problem.
\newblock {\em Mathematical Research Letters}, 16(6):955, 2009.

\bibitem{GU}
A.~Greenleaf and G.~Uhlmann.
\newblock Local uniqueness for the {D}irichlet-to-{N}eumann map via the
  two-plane transform.
\newblock {\em Duke Mathematical Journal}, 108(3):599--617, 2001.

\bibitem{HS}
D.~Hervas and Z.~Sun.
\newblock An inverse boundary value problem for quasilinear elliptic equations.
\newblock {\em Communications in Partial Differential Equations},
  27(11-12):2449--2490, 2002.

\bibitem{H1}
L.~H{\"o}rmander.
\newblock {\em The Analysis of Linear Partial Differential Operators. I,
  Distribution Theory and Fourier Analysis}.
\newblock Grundlehren Der Mathematischen Wissenschaften. Springer, $2^{nd}$
  edition, 1990.

\bibitem{isakov7}
V.~Isakov.
\newblock Completeness of products of solutions and some inverse problems for
  pde.
\newblock {\em Journal of Differential Equations}, 92(2):305--316, 1991.

\bibitem{I1}
V.~Isakov.
\newblock On uniqueness in inverse problems for semilinear parabolic equations.
\newblock {\em Archive for Rational Mechanics and Analysis}, 124(1):1--12,
  1993.

\bibitem{I2}
V.~Isakov.
\newblock Uniqueness of recovery of some quasilinear partial differential
  equations.
\newblock {\em Communications in Partial Differential Equations},
  26(11-12):1947--1973, 2001.

\bibitem{isakov-book}
V.~Isakov.
\newblock {\em Inverse problems for partial differential equations}, volume
  127.
\newblock Springer, $3^{rd}$ edition, 2017.

\bibitem{IN}
V.~Isakov and A.~I. Nachman.
\newblock Global uniqueness for a two-dimensional semilinear elliptic inverse
  problem.
\newblock {\em Transactions of the American Mathematical Society},
  347(9):3375--3390, 1995.

\bibitem{IS}
V.~Isakov and J.~Sylvester.
\newblock Global uniqueness for a semilinear elliptic inverse problem.
\newblock {\em Communications on Pure and Applied Mathematics},
  47(10):1403--1410, 1994.

\bibitem{KKL}
A.~Kachalov, Y.~Kurylev, and M.~Lassas.
\newblock {\em Inverse boundary spectral problems}.
\newblock CRC Press, 2001.

\bibitem{KN}
H.~Kang and G.~Nakamura.
\newblock Identification of nonlinearity in a conductivity equation via the
  {D}irichlet-to-{N}eumann map.
\newblock {\em Inverse Problems}, 18(4):1079, 2002.

\bibitem{KSa}
C.~Kenig and M.~Salo.
\newblock The {C}alder{\'o}n problem with partial data on manifolds and
  applications.
\newblock {\em Analysis \& PDE}, 6(8):2003--2048, 2014.

\bibitem{KV}
{R. V.} Kohn and M.~Vogelius.
\newblock Identification of an unknown conductivity by means of measurements at
  the boundary.
\newblock In {David W.} McLaughlin, editor, {\em SIAM - AMS Proceedings}, SIAM
  - AMS Proceedings, pages 113--123. American Mathematical Soc, December 1984.

\bibitem{KU}
K.~Krupchyk and G.~Uhlmann.
\newblock Partial data inverse problems for semilinear elliptic equations with
  gradient nonlinearities.
\newblock {\em arXiv preprint arXiv:1909.08122}, 2019.

\bibitem{KU2}
Katya Krupchyk and Gunther Uhlmann.
\newblock A remark on partial data inverse problems for semilinear elliptic
  equations.
\newblock {\em Proceedings of the American Mathematical Society},
  148(2):681--685, 2020.

\bibitem{LLLS2}
M.~Lassas, T.~Liimatainen, Y.-H. Lin, and M.~Salo.
\newblock Partial data inverse problems and simultaneous recovery of boundary
  and coefficients for semilinear elliptic equations.
\newblock {\em arXiv preprint arXiv:1905.02764}, 1905.

\bibitem{LLLS1}
M.~Lassas, T.~Liimatainen, Y.-H. Lin, and M.~Salo.
\newblock Inverse problems for elliptic equations with power type
  nonlinearities.
\newblock {\em arXiv preprint arXiv:1903.12562}, 2019.

\bibitem{MU}
C.~Munoz and G.~Uhlmann.
\newblock The {C}alder\'{o}n problem for quasilinear elliptic equations.
\newblock {\em Annales de l'Institut Henri Poincar{\'e} C, Analyse non
  lin{\'e}aire}, 2020.

\bibitem{Ralston}
J.~Ralston.
\newblock Gaussian beams and the propagation of singularities.
\newblock {\em Studies in Partial Differential Equations}, 23(206):C248, 1982.

\bibitem{riesz}
M.~Riesz.
\newblock Int{\'e}grales de {R}iemann-{L}iouville et potentiels.
\newblock {\em Acta litterarum ac scientiarum Regiae Universitatis Hungaricae
  Francisco-Josephinae : Sectio scientiarum mathematicarum}, 9:1--42, 1938.

\bibitem{Sh}
R.~Shankar.
\newblock Recovering a quasilinear conductivity from boundary measurements.
\newblock {\em arXiv preprint arXiv:1910.07890}, 2019.

\bibitem{shar}
V.~Sharafutdinov.
\newblock Linearized inverse problem for the {D}irichlet-to-{N}eumann map on
  differential forms.
\newblock {\em Bulletin des sciences mathematiques}, 133(4):419--444, 2009.

\bibitem{S1}
Z.~Sun.
\newblock On a quasilinear inverse boundary value problem.
\newblock {\em Mathematische Zeitschrift}, 221(1):293--305, 1996.

\bibitem{S3}
Z.~Sun.
\newblock Anisotropic inverse problems for quasilinear elliptic equations.
\newblock In {\em Journal of Physics: Conference Series}, volume~12, page 015.
  IOP Publishing, 2005.

\bibitem{S2}
Z.~Sun.
\newblock An inverse boundary-value problem for semilinear elliptic equations.
\newblock {\em Electronic Journal of Differential Equations (EJDE)[electronic
  only]}, 2010:Paper--No, 2010.

\bibitem{SuU}
Z.~Sun and G.~Uhlmann.
\newblock Inverse problems in quasilinear anisotropic media.
\newblock {\em American Journal of Mathematics}, 119(4):771--797, 1997.

\bibitem{SU}
J.~Sylvester and G.~Uhlmann.
\newblock A global uniqueness theorem for an inverse boundary value problem.
\newblock {\em Annals of Mathematics}, pages 153--169, 1987.

\end{thebibliography}
\bibliographystyle{plain}

\end{document}